\documentclass[12pt,reqno]{amsart}
\usepackage{amscd,amssymb,amsthm}
\usepackage{graphicx}

\usepackage{enumerate}
\theoremstyle{plain}
\newtheorem{theorem}{Theorem}
\newtheorem*{intropr}{Proposition}
\newtheorem*{introth}{Main Theorem}
\newtheorem{corollary}[theorem]{Corollary}
\newtheorem{lemma}[theorem]{Lemma}
\newtheorem{proposition}[theorem]{Proposition}
\theoremstyle{remark}

\newcommand{\reel}{\mathbb{R}}

\newcommand{\vf}{\varphi}

\newcommand{\bg}{\medskip\goodbreak}

\newenvironment{enumeratea}{\begin{enumerate}%
	[\upshape (a)]}{\end{enumerate}}

\title[An Optimal Inequality For The Tangent Function]
{An Optimal Inequality For The Tangent Function}
\author[Omran Kouba]{Omran Kouba$^\dag$}
\address{Department of Mathematics \\
Higher Institute for Applied Sciences and Technology\\
P.O. Box 31983, Damascus, Syria.}
\email{omran\_kouba@hiast.edu.sy}
\keywords{inequality, tangent function, power series expansion.}
\subjclass[2010]{26D05.}
\thanks{$^\dag$ Department of Mathematics, Higher Institute for Applied Sciences and Technology.}

\begin{document}
\parindent=0pt
\begin{abstract}
In this note we deal with some inequalities for the tangent function that are valid for
$x$ in $(-\pi/2,\pi/2)$. These inequalities are optimal in the sense that the best values of 
the exponents involved are obtained. \par
\end{abstract}
\smallskip\goodbreak

\parindent=0pt
\maketitle
\section{\sc Introduction }\label{sec1}
\bg
\parindent=0pt
\qquad  The story started when I wanted to provide my students of  ``Basic Calculus'' class, with a way to prove that
\begin{equation}\label{E:eq1}
\lim_{x\to0}\frac{\tan x-x}{x^3}=\frac{1}{3}
\end{equation}
without recourse to any advanced topics or to the L'H\^opital's rule. So, I came up with the following proposition :\bg

\begin{intropr} 
 For every $x\in(0,\pi/2)$ the following inequality holds:
\begin{equation}\label{E:eqpr1}
x+\frac{x^3}{3}<\tan x<x+\frac{\tan^3x}{3}
\end{equation}
\end{intropr}
\bg
\qquad Clearly, the limit in \eqref{E:eq1} follows easily from this Proposition.
 But this was not the end of the story, it was just the beginning
of my investigation. In fact, the inequality \eqref{E:eqpr1}
 means that $3(\tan x-x)$ is somewhere between $x^3$ and $\tan^3x$, but where exactly ?
\bg
 \qquad In order to describe our results, an important role is played by 
 the family of functions $(f_\gamma)_{\gamma\in[0,3]}$ defined on $[0,\pi/2)$ by 
\begin{equation}\label{E:fdef}
f_\gamma(x)=x^{3-\gamma}\tan^{\gamma}x.
\end{equation}
\bg
\qquad Because of the well-known
inequality $\tan x\geq x$ for $0\leq x<\pi/2$, we see that the family 
$(f_\gamma)_{\gamma\in[0,3]}$ is increasing in the sense that 
$f_\alpha\leq f_\beta$ for $\alpha<\beta$. Using this family, 
we can reformulate \eqref{E:eqpr1} by saying that
\[
f_0(x)<3(\tan x-x)<f_3(x),\quad\hbox{ for $0<x<\frac{\pi}{2}$},
\] 

\qquad So, it is natural to be interested in
 identifying the best $\alpha$ and $\beta$ such that $f_\alpha(x)<3(\tan x-x)<f_\beta(x)$ for $0<x<\pi/2$. 
We were able to completely answer this question, our results are summarized in the following two 
statements :
\bg
\begin{intropr}
If for every $x\in(0,\pi/2)$ we have $f_\alpha(x)\leq 3(\tan x-x)\leq f_\beta(x)$,  where
 $f_{\gamma}$ is defined in \eqref{E:fdef},  then
$\alpha\leq1$ and $\beta\geq 6/5$.
\end{intropr}
\bg
\begin{introth}  The following two inequalities hold :
\begin{enumeratea}
\item For every $x\in(0,\pi/2)$ we have $f_1(x)<3(\tan x-x)$, 
\item For every $x\in(0,\pi/2)$ we have $3(\tan x-x)<f_{6/5}(x)$. 
\end{enumeratea}
where  $f_{\gamma}$ is defined in \eqref{E:fdef}. Equivalently,
\begin{equation}\label{E:eqth}
\forall\,x\in\left(0,\frac{\pi}{2}\right),
\qquad x+\frac{1}{3}x^2\tan x<\tan x<x+\frac{1}{3}x^{9/5}\tan^{6/5} x.
\end{equation}
\end{introth}
\bg
\qquad Before we embark in the proof of our results, it is worth mentioning that there is
a lot of similar inequalities involving trigonometric functions in the literature \cite{bec,guo1,guo2,Qi}. For instance, the Becker-Stark's inequality
\cite{bec} states that
\[
\frac{8x}{\pi^2-4x^2}<\tan x<\frac{\pi^2 x}{\pi^2-4x^2},\qquad \hbox{for $0<x<\frac{\pi}{2}$}.
\]
\bg
\qquad Also, in \cite{Qi} the authors prove, among other things, that for  $0<x<\frac{\pi}{2}$, one has
\begin{equation}\label{E:eqQi}
x+\frac{x^3}{3}+\frac{2}{15}x^4\tan x<\tan x<
x+\frac{x^3}{3}+\left(\frac{2}{\pi}\right)^4x^4\tan x.
\end{equation}
\bg
\qquad Numerical evidence shows that the upper inequality in \eqref{E:eqth} is sharper than the upper
inequality in \eqref{E:eqQi} for $x\in(0,x_0)$ where $x_0\approx 1.2332$, and that the lower
 inequality in \eqref{E:eqth} is sharper than the lower
inequality in \eqref{E:eqQi} for $x\in(x_1,\pi/2)$ where $x_1\approx 1.5255$. So the two results are complementary but not
comparable.

\bigskip\goodbreak
\section{\sc Results and Proofs}\label{sec2}
\bg
\qquad Clearly, the next Proposition \ref{pr1} follows from our main Theorem \ref{th3}, 
but it can be elementarily proved directly, our aim from presenting the proof is just
to compare the degree of difficulty.
\bg
\begin{proposition} \label{pr1}
 For every $x\in(0,\pi/2)$ the following inequality holds:
\begin{equation*}
x+\frac{x^3}{3}<\tan x<x+\frac{\tan^3x}{3}
\end{equation*}
\end{proposition}
\begin{proof}
Indeed, let $g$ and $h$ be the functions defined on $[0,\pi/2)$ by
\begin{align*}
g(x)&=\tan x-x-\frac{x^3}{3},\\
h(x)&=\frac{\tan^3x}{3}+x-\tan x.
\end{align*}
\qquad Clearly,  for $x\in(0,\pi/2)$, we have $g'(x)=\tan^2x-x^2>0$ and $h'(x)=\tan^4x>0$. Thus, both $g$ and $h$
are monotonous increasing on the interval $(0,\pi/2)$, and the desired inequality follows since
$g(0)=h(0)=0$.
\end{proof}
\bg
\bigskip\goodbreak

\begin{proposition} \label{pr2}
If for some $0\leq\alpha,\beta\leq3$, we have
\[
\left(\frac{\tan x}{x}\right)^\alpha\leq \frac{3(\tan x-x)}{x^3}\leq \left(\frac{\tan x}{x}\right)^\beta
\]
for every $x\in(0,\pi/2)$, then $\alpha\leq1$ and $\beta\geq 6/5$.
\end{proposition}

\begin{proof}

Suppose that for $x\in(0,\pi/2)$ we have

\[
\left(\frac{\tan x}{x}\right)^\alpha\leq \frac{3(\tan x-x)}{x^3}\leq \left(\frac{\tan x}{x}\right)^\beta,
\]
that is $\alpha\leq \vf(x)\leq \beta$ where $\vf$ is defined on $(0,\pi/2)$ by
\[
\vf(x)=\log\left(\frac{3(\tan x-x)}{x^3}\right)\left/\log\left(\frac{\tan x}{x}\right).\right.
\]
Now, since
\[
\vf(x)=\frac{\log(\tan x)+\log(1-x/\tan x)+\log 3-3\log(x)}{\log(\tan x)-\log x}
\]
we conclude that 
\begin{equation}\label{E:eqlm1}
\lim_{x\to\left(\frac{\pi}{2}\right)^-}\vf(x)=1.
\end{equation}
\bg
On the other hand, since in the neighborhood of $0$ we have
\[\frac{\tan x}{x}=1+\frac{x^2}{3}+\frac{2}{15}x^4+ O(x^6),
\]
we deduce that
\begin{align*}
\log\left(\frac{\tan x}{x}\right)&=\frac{x^2}{3}+O(x^4)\\
\log\left(\frac{3(\tan x-x)}{x^3}\right)&=\log\left(1+\frac{2}{5}x^2+ O(x^4)\right)\\
&=\frac{2}{5}x^2+ O(x^4).
\end{align*}
Thus, $\vf(x)=\frac{6}{5}+O(x^2)$, and consequently 
\begin{equation}\label{E:eqlm2}
\lim_{x\to0^+}\vf(x)=\frac{6}{5}.
\end{equation}
Therefore, \eqref{E:eqlm1} and \eqref{E:eqlm2}, together with the fact that
 $\alpha\leq\vf(x)\leq\beta$ for every $x\in(0,\pi/2)$, 
 imply that $\alpha\leq 1$ and $\beta\geq\frac{6}{5}$ as desired.
\end{proof}
\bg
\qquad Before we come to the proof of our main theorem, we will need the following technical lemma.
\begin{lemma}\label{lm}
Let $\vf $ be the function defined on $\reel$ by
\begin{equation}
\vf(x)=(9-24x^2)\cos(x)-9\cos(3x)-4x \sin(3x).
\end{equation}
Then $\vf(x)>0$ for $0<x\leq \pi/2$.
\end{lemma}
\begin{proof}
In order to determine the sign of $\vf(x)$ for $x\in(0,\pi/2]$, we will use power series expansion. Clearly, for every real $x$ we have
\begin{align*}
\vf(x)&=(9-24x^2)\sum_{n=0}^\infty\frac{(-1)^nx^{2n}}{(2n)!}-9\sum_{n=0}^\infty\frac{(-1)^n3^{2n}x^{2n}}{(2n)!}
-4x\sum_{n=1}^\infty\frac{(-1)^{n-1}3^{2n-1}x^{2n-1}}{(2n-1)!}\\
&=\sum_{n=0}^\infty\big(9+24(2n)(2n-1)-9\cdot 3^{2n}+4(2n)\cdot 3^{2n-1}\big)\frac{(-1)^nx^{2n}}{(2n)!}\\
&=3\sum_{n=0}^\infty\big(32n^2-16n+3)+(8n-27)\cdot 3^{2n-2}\big)\frac{(-1)^nx^{2n}}{(2n)!}
\end{align*}
Thus, for a real $x$ we have
\begin{equation}\label{E:eq26}
\vf(x)=3\sum_{n=0}^\infty(-1)^n\frac{T_n}{(2n)!}x^{2n},
\end{equation}
where,
\begin{equation}\label{E:eq27}
T_n=2(4n-1)^2+1+(8n-27)9^{n-1}.
\end{equation}
Noting that $T_0=T_1=T_2=T_3=0$ we conclude that \eqref{E:eq26} can be written as follows
\begin{equation}\label{E:eq28}
\vf(x)=3\sum_{n=4}^\infty  (-1)^n\frac{T_n}{(2n)!}x^{2n}.
\end{equation}
\qquad We recognize an alternating series since it is clear from \eqref{E:eq27} that $T_n>0$ for $n\geq4$. Now, if we show
that the sequence $\left(\frac{T_n}{(2n)!}x^{2n}\right)_{n\geq4}$ is decreasing for any $x\in(0,\pi/2]$ then 
this would imply that $\vf(x)>0$ for $x\in(0,\pi/2]$, because the first term in the series \eqref{E:eq28} is positive.\bg

\qquad Let $U_n$ be defined by,
\begin{equation}\label{E:eq29}
U_n=(2n+2)(2n+1)T_n-3T_{n+1}.
\end{equation}
a simple calculation shows that
\begin{align}\label{E:eq30}
U_n&=(4n^2+6n+2)\big(32n^2-16n+3 +(8n-27)9^{n-1}\big)\notag\\
&\quad\qquad-3(32n^2+48n+19+(8n-19)9^{n})\notag\\
&= 128n^4+128n^3-116n^2-158n-51+(32n^3-60n^2-362n+459)9^{n-1}\notag\\
&=B_n+A_n\cdot 9^{n-1}
\end{align}
where
\begin{align*}
B_n&=128n^4+128n^3-116n^2-158n-51\\
A_n&=32n^3-60n^2-362n+459
\end{align*}
Now, it is straightforward to check that
\begin{align*}
B_{n+1}&=128n^4+640n^3+1036n^2+437n+69(n-1),\\
A_{n+4}&=32n^3+324n^2+694 n+99.
\end{align*}
Thus, $A_n$ and $B_n$ are positive for $n\geq4$, and according to \eqref{E:eq30} we have $U_n>0$ for  $n\geq4$.
Using \eqref{E:eq29} we conclude that for $n\geq4$ and $x\in(0,\sqrt{3}]$ we have
\begin{equation*}
(2n+2)(2n+1)T_n>x^2T_{n+1}
\end{equation*} 
or, equivalently,
\begin{equation*}
\forall\,n\geq4,\quad\forall\,x\in\left(0,\sqrt{3}\right],\quad\frac{T_n}{(2n)!}x^{2n}>\frac{T_{n+1}}{(2n+2)!}x^{2n+2}.
\end{equation*} 
It follows that the sequence $\left(\frac{T_n}{(2n)!}x^{2n}\right)_{n\geq4}$ is decreasing for any $x\in(0,\sqrt{3}]$, and, as we have
already explained, this implies using \eqref{E:eq28} that $\vf(x)>0$ for  $x\in\left(0,\sqrt{3}\right]$, and the Lemma follows
since $\frac{\pi}{2}<\sqrt{3}$.
\end{proof}
\bg
\qquad With this technical lemma at hand, we can prove our Main Theorem.
\bg

\begin{theorem} \label{th3} The following two inequalities hold :
\begin{enumeratea}
\item For every $x\in(0,\pi/2)$ we have $x^2\tan x<3(\tan x-x)$, 
\item For every $x\in(0,\pi/2)$ we have $3(\tan x-x)<x^{9/5}(\tan x)^{6/5}$. 
\end{enumeratea}
\end{theorem}
\bg

\begin{proof} 
\textbf{(a)} Consider the function $g$ defined on the interval $(0,\pi/2)$ by
\begin{equation}\label{E:eq21}
g(x)=3-x^2-3x\cot x
\end{equation}
Clearly we have
\begin{equation}\label{E:eq22}
g'(x)= x-3\cot x+3x \cot^2x=(1+3\cot^2x)h(x)
\end{equation}
where $h(x)=x-\dfrac{3\tan x}{3+\tan^2 x}$. Similarly $h$ has a derivative on $[0,\pi/2)$
that is given by
\begin{align*}
h'(x)&=1-3\frac{(3-\tan^2x)(1+\tan^2x)}{(3+\tan^2x)^2}\\
&=\frac{4\tan^2x}{(3+\tan^2x)^2}
\end{align*}
So, $h$ is monotonous increasing, with $h(0)=0$. This implies that $h$ is positive on the interval
$(0,\pi/2)$. Going back to \eqref{E:eq22} we conclude that
$g$ is also  monotonous increasing on $(0,\pi/2)$. Finally, since
$\lim_{x\to0^+}g(x)=0$, we conclude that $g$ is positive on $(0,\pi/2)$, but it is straightforward
to check that this is equivalent to the fact that $3(\tan x-x)> x^2\tan x $ for
$x\in (0,\pi/2)$ which is the desired inequality.\bg

\textbf{(b)} This inequality is more delicate to prove.
 Again, we will consider an auxiliary function. Let $g$ be the function defined on $(0,\pi/2)$ by
\begin{equation}\label{E:eq23}
g(x)=6\log\left(\frac{\tan x}{x}\right)-5\log\left(\frac{3(\tan x-x)}{x^3}\right).
\end{equation}

Clearly we have
\begin{align*}
g'(x)&=\frac{6}{\cos x\,\sin x}+\frac{9}{x}-\frac{5\sin^2x}{\cos x(\sin x-x\cos x)}\\
&=\frac{(9-6x^2)\cos x+x(4\sin^3x-3\sin x)-9\cos^3x}{x \cos x\,\sin x\, (\sin x-x\cos x)}
\end{align*}
So, recalling the expression of $\cos(3x)$ and $\sin(3x)$ in terms of $\cos x$ and $\sin x$ we see that
\begin{align}\label{E:eq24}
g'(x)&=\frac{(9-24x^2)\cos(x)-9\cos(3x)-4x \sin(3x)}{4x \cos^2 x\,\sin x\, (\tan x-x)},\notag\\
 &=\frac{\vf(x)}{4x \cos^2 x\,\sin x\, (\tan x-x)},
\end{align}
where $\vf$ is the function considered in Lemma \ref{lm}. Using the conclusion of that Lemma  we see that
$g$ is monotonous increasing on $(0,\pi/2)$. But $\lim_{x\to0^+} g(x)=0$, so $g$ is positive 
on $(0,\pi/2)$, and this
is equivalent to $3(\tan x-x)<x^{9/5} \tan^{6/5}x$ which is the desired inequality.
\end{proof}
\bg

\begin{corollary}\label{cor1}
The necessary and sufficient condition, on the real numbers $\alpha$ and $\beta$, for the following inequality
\begin{equation*}
1+\frac{x^2}{3}\left(\frac{\tan x}{x}\right)^\alpha<\frac{\tan x}{x}<1+\frac{x^2}{3} \left(\frac{\tan x}{x}\right)^{\beta}
\end{equation*}
to hold for every nonzero real $x$ from $(-\pi/2,\pi/2)$, is that $\alpha\leq1$ and $\beta\geq 6/5$.  
\end{corollary}
\begin{proof}
This follows from Proposition \ref{pr2}, Theorem \ref{th3}, and from the fact that the considered functions are even.
\end{proof}



\begin{thebibliography}{9}
\setlength{\itemsep}{5pt}

\bibitem{bec}
Becker,~M. and Stark,~E. ~L.,
\emph{ On a hierarchy of quolynomial inequalities for $\tan x$},
 Univ. Beograd. Publ. Elektrotechn. Fak. Ser. Mat. Fiz., No. \textbf {602-633}(1978), 133--138.

\bibitem{guo1}
Guo,~B.-N., Li,~W and Qi,~F.,
\emph{Proofs of Wilker's inequalities involving trigonometric functions.},
 Inequality Theory and Applications, \textbf {2},  (2003), 109--112.
Nova Science Publichers.

\bibitem{guo2}
Guo,~B.-N., Qiao,~B.-M.,  Qi,~F., and Li,~W
\emph{On new proofs of inequalities involving trigonometric functions.},
Math. Inequal. Appl., \textbf {6},  No.1 (2003), 19--22.

\bibitem{Qi}
Chen,~Ch.-P. and Qi, F.,
\emph{A Double Inequality for Remainder of Power Series of Tangent Function},
 Tamkang Journal of Mathematics, \textbf {34}, No. 4, (2003), 351--355.

\end{thebibliography}
\end{document}